%% file: main.tex
\documentclass[runningheads]{llncs}

\usepackage{color}

\usepackage[
    draft = false,
    unicode = true,
    colorlinks = true,
    allcolors = blue,
    hyperfootnotes = true,
    citecolor = red
]{hyperref}

\usepackage[T1, T2A]{fontenc}
\usepackage[russian,english]{babel}
\usepackage{algorithmic}
\usepackage{graphicx}

\renewcommand{\geq}{\geqslant}

\renewcommand{\leq}{\leqslant}

\input{macros}

\begin{document}

\title{Nesterov's method of dichotomy via Order Oracle: The problem of optimizing a two-variable function on a square}
\titlerunning{Nesterov's method of dichotomy via Order Oracle}

\author{Boris Chervonenkis\inst{1}\orcidID{0009-0000-4138-8337} \and
Andrei Krasnov\inst{1}\orcidID{0009-0006-4910-0944} \and
Alexander Gasnikov\inst{1, 2, 3}\orcidID{0000-0002-7386-039X} \and
Aleksandr Lobanov\inst{1, 4, 5}\orcidID{0000-0003-1620-9581} }
\authorrunning{B. Chervonenkis et al.}
%
\institute{Moscow Institute of Physics and Technology, Russia \and
Institute for Information Transmission Problems RAS, Moscow, Russia \and
Innopolis University, Russia \and
Skolkovo Institute of Science and Technology, Russia \and
The Institute for System Programming of the Russian Academy of Sciences, Russia }

\maketitle

\begin{abstract}
    The challenges of black box optimization arise due to imprecise responses and limited output information. This article describes new results on optimizing multivariable functions using an Order Oracle, which provides access only to the order between function values and with some small errors. We obtained convergence rate estimates for the one-dimensional search method (golden ratio method) under the condition of oracle inaccuracy, as well as convergence results for the algorithm on a "square" (also with noise), which outperforms its alternatives. The results obtained are similar to those in problems with oracles providing significantly more information about the optimized function. Additionally, the practical application of the algorithm has been demonstrated in maximizing a preference function, where the parameters are the acidity and sweetness of the drink. This function is expected to be convex or at least quasi-convex.
\end{abstract}

\section{Introduction}

The search for the minimum of a function is a common problem. In optimization, the concept of a "black box" \cite{chen2017zoo} is traditionally used, where providing an input point yields the function value (zeroth-order oracle) or gradient (first-order oracle) at that point. However, when dealing with human feedback, such information is often unavailable. In such cases, one can ask, "Which point do you prefer, x or y?" Human preferences may be imprecise, introducing small errors known as noise. This leads us to the concept of an order oracle.

This particular type of oracle has been receiving special attention for some time now. The Stochastic Three Points Method proposed by Bergou et al. (2020) \cite{bergou2020stochastic} uses an oracle that compares three function values at once and achieves a complexity of $O(dL/\mu \,\, log(1/\varepsilon))$ iterations in the strongly convex case (notations are defined here \ref{sec:def}). Saha et al. (2021) \cite{saha2021dueling} introduced another algorithm where the oracle compares two function values only once per iteration. The analysis of this algorithm is based on the Sign SGD and achieves a complexity of $O(dL/\mu \,\, log(1/\varepsilon))$ iterations in the strongly convex case. Tang et al.'s work (2023) \cite{tang2023zeroth} demonstrated the extensive use of the Order Oracle in Reinforcement Learning with Human Feedback, providing an estimate of oracle complexity as $O(d/\varepsilon^2)$ iterations specifically in the non-convex case. There have been many other notable results in this field \cite{jamieson2012query,gorbunov2019stochastic,lobanov2024order}.

There has long been an algorithm that can work with order oracle in one-dimensional space: the golden ratio algorithm (GRM). There were also attempts to generalize it to two-dimensional space \cite{he2002object}.

Our proposed algorithm uses a line search via GRM and is particularly effective for small-scale problems even in the presence of noise. For sufficiently large $\varepsilon$ on a square with side length $R$ and noise $\Delta$, its complexity is \\ $O(log(MR/\varepsilon) \,\, log(MR/\Delta))$ iterations.

Moreover, we demonstrate its application in the practical task of maximizing the preference function for a sour-sweet beverage.

\section{Definitions and designations} \label{sec:def}
$f : D \rightarrow \mathbb{R}, D \in \mathbb{R}^d$ --- a function the minimum of which needs to be found. We will assume that it is convex and differentiable (otherwise the problem immediately becomes much more difficult).\\
\\
Let's formalize the oracle order. It will be a function $$\psi(x, y) = sign(f(x) - f(y) + \delta(x, y)),$$ where $\delta(x, y)$ is  a real noise function, which we consider bounded: $|\delta(x, y)| \leqslant \Delta$ for all $x$, $y$.\\
\\
$\varepsilon$ --- accuracy in the function value with which we want to evaluate the point of minimum. In other words, we are looking for such an $x_0$, that $f(x_0) - minf(x) \leqslant \varepsilon$.\\
\\
We will also assume that the function $f$ has Lipschitz properties:
\begin{enumerate}
    \item Lipschitzness with respect to the argument: $|f(x) - f(y)| \leqslant M \cdot \|x - y\|$ for all $x$, $y$.
    \item Lipschitzness with respect to the gradient: $|\nabla f(x) - \nabla f(y)| \leqslant L \cdot \|x - y\|$ for all $x$, $y$.
\end{enumerate}
Sometimes we will need to consider a narrower class of functions --- strongly convex ones:\\
a function $f$ is strongly convex with a coefficient $\mu$, if for all $x$, $y$ the inequality is satisfied $$f(y) \geqslant f(x) + \langle \nabla f(x), y - x \rangle + \frac{\mu}{2}\|y-x\|^2 $$
Note that $\Delta, \varepsilon, M, L$, as well as in case of strong convexity, $\mu$ are positive numbers and we consider them known unlike the function $f$ and the noise function $\delta$ themselves.

\section{Main algorithms}
    Let's start by considering the simplest case --- one-dimensional: $d = 1$, that is, when $f$ is a function of a single variable. This case is very important, as all further algorithms will be based on it. The fastest algorithm here is the so-called golden ratio method \ref{sec:GRM}.
    
    Next is the general case, where $f$ is a function of multiple variables. There are already many possible algorithms, such as the coordinate descent method, accelerated methods, and others. A detailed discussion of them with consideration of different problem formulations can be found here \cite{lobanov2024order}.
    
    There is also an algorithm that works well in low dimensions (here considered in a two-dimensional variant) based on Nesterov's method \ref{sec:Square}. The original algorithm, which works for a first-order oracle (i.e., capable of computing the value and gradient of the function at a point), can be found here \cite{пасечнюк2019одном}.

\section{Golden ratio method}\label{sec:GRM}
    \emph{Algorithm description}. Let's find the minimum of a convex function \( f \) with good accuracy in the interval \( [a, b] \). Let \( a_0 = a \), \( b_0 = b \), \( s_0 = b - \frac{{b-a}}{{\phi}} \), \( t_0 = a + \frac{{b-a}}{{\phi}} \). Thus, each of the marked points \( s \) and \( t \) divides the interval \( [a, b] \) in the golden ratio: \( \frac{{b_0-s_0}}{{s_0-a_0}} = \phi \), \( \frac{{t_0-a_0}}{{b_0-t_0}} = \phi \). The current interval on which we are trying to find the minimum is \( [a_0, b_0] \). After \( i \) iterations of the algorithm, we have an interval \( [a_i, b_i] \), with points \( s_i \) and \( t_i \) marked on it, dividing it in the golden ratio: \( \frac{{b_i-s_i}}{{s_i-a_i}} = \phi \), \( \frac{{t_i-a_i}}{{b_i-t_i}} = \phi \). We request the comparison result of the function values at points \( s_i \) and \( t_i \), that is \( \psi(s_i, t_i) \). If it turns out that \( \psi(s_i, t_i) = "+" \), then \( f(s_i) \) is greater or insignificantly less than \( f(t_i) \) (an erroneous result can occur due to noise), and in this case we "discard" the interval \( [a, s_i] \), moving on to consider the interval \( [s_i, b] \) (assuming that the minimum is not in \( [a, s_i] \)). Thus, \( a_{i+1} = s_i \), \( b_{i+1} = b_i \), \( s_{i+1} = b_{i+1} - \frac{{b_{i+1}-a_{i+1}}}{{\phi}} \), \( t_{i+1} = a_{i+1} + \frac{{b_{i+1}-a_{i+1}}}{{\phi}} \). Note that \( s_{i+1} \) and \( t_{i+1} \) could be recalculated as follows: \( s_{i+1} = t_i \), \( t_{i+1} = a_{i+1} + b_{i+1} - s_{i+1} \). If \( \psi(s_i, t_i) = "-" \), by analogy, \( a_{i+1} = a_i \), \( b_{i+1} = t_i \), \( t_{i+1} = s_i \), \( s_{i+1} = a_{i+1} + b_{i+1} - t_{i+1} \). After performing \( n \) such iterations, we take \( \frac{{a_n+b_n}}{2} \) as the answer.

\newpage

\begin{algorithm}[H]
    \caption{Golden Ratio Method}
    \begin{algorithmic}[1]
        \STATE $\phi \gets \frac{\sqrt{5} + 1}{2}$
        \STATE $a_0 \gets a$
        \STATE $b_0 \gets b$
        \STATE $s_0 \gets b - \frac{{b-a}}{{\phi}}$
        \STATE $t_0 \gets a + \frac{{b-a}}{{\phi}}$
	\FOR {$i=0,1,...,n-1$}
            \IF {$\psi(s_i, t_i) = "+"$}
                \STATE $a_{i+1} \gets s_i$
                \STATE $b_{i+1} \gets b_i$
                \STATE $s_{i+1} \gets t_i$
                \STATE $t_{i+1} \gets a_{i+1} + b_{i+1} - s_i$
            \ENDIF
            \IF {$\psi(s_i, t_i) = "-"$}
                \STATE $a_{i+1} \gets a_i$
                \STATE $b_{i+1} \gets t_i$
                \STATE $t_{i+1} \gets s_i$
                \STATE $s_{i+1} \gets a_{i+1} + b_{i+1} - t_i$
            \ENDIF
	\ENDFOR
        \STATE $answer = \frac{a_n + b_n}{2}$
    \end{algorithmic}
\end{algorithm}

It turns out that by performing not so many iterations, we quickly obtain a fairly accurate answer. Let's make some estimates on the optimal number of steps and the accuracy of the found solution.

\begin{theorem} \label{th:1}
    Let the function $f$ be defined on the interval $[0, R]$. The golden ratio $\phi = \frac{\sqrt{5} + 1}{2}$, $e$ --- Euler's number. Let $C = \frac{e R M \ln{\phi}}{2 \phi}$. Then, if we perform $n_0 = \log_\phi(\frac{C}{e\Delta})$ iterations of the algorithm, the error in terms of the value of the function at the midpoint of the remaining interval will be not worse than $\phi\Delta \log_\phi(\frac{C}{\Delta})$.
\end{theorem}

\begin{proof}
    Let $[a_k, b_k]$ be the residual interval after $k$ iterations. $[a_0, b_0] = [0, R]$. It is easy to see that $b_k - a_k = \frac{R}{\phi^k}$. Let $m_k = \min\limits_{x \in [a_k, b_k]}f(x)$. Consider the next iteration of the algorithm. At this point, the residual interval is $[a_k, b_k]$. Let the points where values are compared be denoted as $s$ and $t$ ($a_k < s < t < b_k$). We aim to prove that
    \begin{equation} \label{eq:1}
        m_{k+1} - m_k \leqslant \phi\Delta.
    \end{equation}
    Without loss of generality, assume that the oracle claims $f(s) > f(t)$. Then there are two cases:
    \begin{enumerate}
        \item The oracle is correct. In this case, $m_{k+1} = m_k$, and inequality \eqref{eq:1} is obvious.
        \item The oracle is incorrect. Then $0 \leqslant f(t) - f(s) \leqslant \Delta$. The new residual interval will be $[s, b_k]$. By convexity, we have $$\forall x \in [a_k, s] \quad f(x) \geqslant f(s) - \frac{s - x}{t - s} (f(t) - f(s)) \geqslant f(s) - \frac{s - a_k}{t - s}\Delta = f(s) - \phi\Delta,$$ which implies $m_k \geqslant m_{k+1} - \phi\Delta$, and inequality \eqref{eq:1} is proven.
    \end{enumerate}
    Summing up inequality \eqref{eq:1} over $k$, we obtain that after $n$ iterations of the algorithm
    \begin{equation} \label{eq:2}
        m_n \leqslant n\phi\Delta + m_0.
    \end{equation}
    The midpoint $x_0$ is chosen on the interval $[a_n, b_n]$. By Lipschitzness with respect to the argument, we have
    $$f(x_0) - m_0 \leqslant \frac{b_n - a_n}{2} M = \frac{R M}{2 \phi^n}.$$
    Consequently, using \eqref{eq:2}, we get
    \begin{equation} \label{eq:3}
        f(x_0) - m_0 \leqslant f(x_0) - m_n + n\phi\Delta \leqslant \frac{RM}{2\phi^n} + n\phi\Delta.
    \end{equation}
    By performing $n_0 = \log_\phi(\frac{C}{e\Delta}) = \log_\phi(\frac{RM\ln{\phi}}{2\phi\Delta})$ iterations, we achieve an accuracy of
    $$\frac{RM}{2\phi^n} + n\phi\Delta = \frac{\phi\Delta}{\ln{\phi}} + \log_\phi(\frac{C}{e\Delta})\phi\Delta = \phi\Delta\log_\phi(\frac{C}{\Delta})$$
\end{proof}

\begin{lemma} \label{l:1}
    Let the function $f$ be defined on the interval $[0, R]$, and let it be strongly convex with a coefficient $\mu$. Suppose a point $x_0$ has been found with accuracy $\varepsilon$ in terms of function value. Then $x_0$ is a minimum with accuracy $\delta = \sqrt{\frac{2\varepsilon}{\mu}}$ in terms of the argument, that is, $x_0 - \arg\min{f(x)} \leqslant \delta$.
\end{lemma}

\begin{proof}
    Let $x^*$ be the point of minimum. From strong convexity, we have $$f(x_0) \geqslant f(x^*) + \langle \nabla f(x^*), x_0 - x^* \rangle + \frac{\mu}{2}\|x^*-x_0\|^2 = f(x^*) + \frac{\mu}{2}|x_0 - x^*|^2,$$ since $\nabla f(x^*) = 0$. Manipulating the inequality, we obtain $$\varepsilon \geqslant f(x_0) - f(x^*) \geqslant \frac{\mu}{2}|x_0 - x^*|^2.$$ Hence, we conclude that $x_0 - x^* \leqslant \sqrt{\frac{2\varepsilon}{\mu}}$, and thus the lemma is proven.
\end{proof}

\section{Method for finding the minimum on a square}\label{sec:Square}
    Let's move on to the two-dimensional case. Suppose the function $f$ is defined on a square $\Pi \subset \mathbb{R}^2$ with side length $R$.\\
    \\
    \emph{Algorithm description}. On the horizontal midline of the square, we find the minimum point $A$ (with some level of accuracy). This line divides the square into two rectangles. Then, on the vertical line passing through $A$ we also find the minimum point --- $B$. We keep the rectangle into which point $B$ falls and discard the remaining one. We repeat a similar procedure with the remaining rectangle, first dividing it by a vertical midline and then by a horizontal one, which means after pruning, we are left with a square half the size of the original one. We repeat these procedures a certain number of times and then take the central point of the remaining small square.


\begin{algorithm}[H]
    \caption{Two-dimensional Search}
    \begin{algorithmic}[1]
        \STATE $LinearSearch((x_1, y_1), (x_2, y_2))$ --- result of one-dimensional search between points $(x_1, y_1)$ and $(x_2, y_2)$ (it is a point)
        \STATE $a$ --- x-coordinate of the center of the square
        \STATE $b$ --- y-coordinate of the center of the square
        \STATE $d$ --- half the side of the square
	\FOR {$i=0,1,...,n-1$}
            \STATE $result_1 \gets LinearSearch((a-d, b), (a+d, b))$
            \STATE $result_2 \gets LinearSearch((result_1.x, b-d), (result_1.x, b+d))$
            \IF {$result_2.y \leqslant result_1.y$}
                \STATE $b \gets b + \frac{d}{2}$
            \ELSE
                \STATE $b \gets b - \frac{d}{2}$
            \ENDIF
            \STATE $result_1 \gets LinearSearch((a, b-\frac{d}{2}), (a, b+\frac{d}{2}))$
            \STATE $result_2 \gets LinearSearch((a-d, result_1.y), (a+d, result_1.y))$
            \IF {$result_2.x \leqslant result_1.x$}
                \STATE $a \gets a + \frac{d}{2}$
            \ELSE
                \STATE $a \gets a - \frac{d}{2}$
            \ENDIF
	\ENDFOR
        \STATE $answer = (a, b)$
    \end{algorithmic}
\end{algorithm}

\begin{theorem} \label{th:2}
    Let the auxiliary one-dimensional problems be solved with accuracy
    \begin{equation} \label{eq:4}
        \delta = \frac{\varepsilon}{2(2+\sqrt{10})LR}.
    \end{equation}
    Then, we are guaranteed to obtain a point with the desired accuracy ($\varepsilon$)  in terms of the function value by performing
    \begin{equation} \label{eq:5}
        n = \log_2\frac{MR\sqrt{2}}{\varepsilon}
    \end{equation}
    iterations (transition from one square to a square half the size counts as one iteration).
\end{theorem}

\begin{proof}
    \quad
    \begin{enumerate}
        \item Let's denote the remaining square after $j$ iterations as $K_j$ ($K_0 = \Pi$). Since each subsequent square is half the size of the previous one, the diagonal of $K_n$ is $\frac{R\sqrt{2}}{2^n}$.  Then, if $x$ is the center of $K_n$, we have
        \begin{equation} \label{eq:6}
            f(x) - \min\limits_{x \in K_n}f(x) \leqslant M\|x - \arg\min\limits_{x \in K_n}f(x)\| \leqslant M\frac{R\sqrt{2}}{2^n}\frac{1}{2} \stackrel{\eqref{eq:5}}{=} \frac{\varepsilon}{2}.
        \end{equation}
        Now, if we can show that
        \begin{equation} \label{eq:7}
            \min\limits_{x \in K_n}f(x) - \min\limits_{x \in \Pi}f(x) \leqslant \frac{\varepsilon}{2},
        \end{equation}
        by adding \eqref{eq:6} and \eqref{eq:7}, we obtain the required inequality.

        \item After $j$ iterations, the square $K_j$ with a side length of $R_j = \frac{R}{2^j}$ remains. At the current iteration, it is divided into two rectangles --- $Q$ and $Z$, and $Q$ remains. Let's denote the horizontal midline as $l_1$. A point $x_\delta$ (corresponding to point $A$ in the algorithm description) is found on this line. Let $x_0$ be the point where the minimum on $l_1$ is achieved. Then $\|x_\delta - x_0\| \leqslant \delta$. The line passing vertically through $x_\delta$ is denoted as $l_2$. A point $x_\delta$ (corresponding to point $B$ in the algorithm description) is found on this line. Let $x'_0$ be the point where the minimum on $l_2$ is achieved. Then $\|x'_\delta - x'_0\| \leqslant \delta$. Let's prove that
        \begin{equation} \label{eq:8}
            \min\limits_{x \in Q}f(x) - \min\limits_{x \in K_j}f(x) \leqslant L\delta\frac{R_j\sqrt{10}}{2}.
        \end{equation}
        Assume that \eqref{eq:8} does not hold. Then 
        $x_* :=  \arg\min\limits_{x \in K_j}f(x)$ lies in $Z$. We will consider two cases of the location of $x'_0$.
        \begin{enumerate}
            \item $x'_0 \in Q$.\\
            Since $x_0$ is the minimum point on $l_1$, and $f$ is convex and achieves a minimum on $Z$, the gradient vector $\nabla f(x_0) \perp l_1$ and points towards Q. Additionally, since $x'_0$ --- is the minimum point on $l_2$, and $x'_0 \in Q$, the gradient vector $\nabla f(x_\delta)$ projected onto $l_2$ points towards $Z$, thus it itself points towards $Z$. The angle between the gradient vectors at points $x_0$ and $x_\delta$ is not acute. Using the fact that the largest side of a triangle is opposite to the obtuse angle, we can write the following inequality:
            \begin{equation} \label{eq:9}
                \|\nabla f(x_0)\| \leqslant \|\nabla f(x_\delta) - \nabla f(x_0)\| \leqslant L\|x_\delta - x_0\| \leqslant L\delta.
            \end{equation}
            Due to the convexity of $f$ and $x_*$ lying in $Z$, $x_0$ is the minimum over all of $Q$. Also note that $\|x_* - x_0\|$ is not greater than the diagonal of $Z$, i.e., $\frac{R_j\sqrt{5}}{2}$. Therefore, we obtain
            \begin{align*}
                \min\limits_{x \in Q}f(x) - \min\limits_{x \in K_j}f(x) &= f(x_0) - f(x_*) \leqslant \langle \nabla f(x_0), x_* - x_0 \rangle \\ &\leqslant \|\nabla f(x_0)\| \cdot \|x_* - x_0\| \stackrel{\eqref{eq:9}}{\leqslant} L\delta\frac{R_j\sqrt{5}}{2}.
            \end{align*}
            Thus, \eqref{eq:8} holds. Contradiction.
            
            \item $x'_0 \in Z$.\\
            By similar reasoning, we have $\nabla f(x_0) \perp l_1 \perp l_2 \perp \nabla f(x'_0)$. Since $Q$ is the rectangle that remains, $x'_\delta \in Q$. However, $x'_0 \in Z$. Therefore, $\|x_\delta - x'_0\| \leqslant \delta$. Then
            \begin{equation} \label{eq:10}
                \|x_0 - x'_0\| = \sqrt{\|x_\delta - x'_0\|^2 + \|x_\delta - x_0\|^2} \leqslant \sqrt{2}\delta.
            \end{equation}
            The gradient vectors at points $x_0$ and $x'_0$ are perpendicular, so
            \begin{equation} \label{eq:11}
                \|\nabla f(x_0)\| \leqslant \|\nabla f(x_0) - \nabla f(x'_0)\| \leqslant L\|x_0 - x'_0\| \stackrel{\eqref{eq:10}}{\leqslant} L\sqrt{2}\delta.
            \end{equation}
            Now, similarly to the previous case, we get
            \begin{align*}
                \min\limits_{x \in Q}f(x) - \min\limits_{x \in K_j}f(x) &= f(x_0) - f(x_*) \leq \langle \nabla f(x_0), x_* - x_0 \rangle\\
                &\leq \|\nabla f(x_0)\| \cdot \|x_* - x_0\| \stackrel{\eqref{eq:11}}{\leq} L\delta\frac{R_j\sqrt{10}}{2}.
            \end{align*}
            And once again, \eqref{eq:8} holds. Contradiction.
        \end{enumerate}
        Thus, \eqref{eq:8} is proven.

        \item Similarly, we can obtain an estimate for the difference in minima when transitioning from $Q$ to $K_{j+1}$. The only change will be the maximum distance between points within one region - it will decrease from $\frac{R_j\sqrt{5}}{2}$ to $\frac{R_j\sqrt{2}}{2}$. Thus, we obtain an estimate analogous to \eqref{eq:8}:
        \begin{equation} \label{eq:12}
            \min\limits_{x \in Q}f(x) - \min\limits_{x \in K_j}f(x) \leqslant L\delta R_j.
        \end{equation}

        \item Using \eqref{eq:8} and \eqref{eq:12}, we have
        $$\min\limits_{x \in K_n}f(x) - \min\limits_{x \in K_0}f(x) = \sum_{j=0}^{n-1} \min\limits_{x \in K_{j+1}}f(x) - \min\limits_{x \in K_j}f(x) \leqslant$$ $$\leqslant \sum_{j=0}^{n-1} L\delta\frac{R_j\sqrt{10}}{2} + L\delta R_j = L\delta\sum_{j=0}^{n-1} \frac{R}{2^j}(\frac{\sqrt{10}}{2} + 1) \leqslant L\delta R(2+\sqrt{10}) \stackrel{\eqref{eq:4}}{=} \frac{\varepsilon}{2}$$
        We have obtained \eqref{eq:7}, which concludes the proof of the theorem.
    \end{enumerate}
\end{proof}

\quad\\
Now we can combine the results obtained in one-dimensional and two-dimensional cases.

\begin{theorem} \label{th:3}
    Let the function $f$ be defined on the square $\Pi \subset \mathbb{R}^2$ with side length $R$, where it is strongly convex with a coefficient $\mu$. And let $C = \frac{e R M \ln{\phi}}{2 \phi}$. Then in the algorithm from the theorem \ref{th:2}, under the condition that $$\varepsilon \geqslant \sqrt{\frac{2\phi\Delta\log_\phi(\frac{C}{\Delta})}{\mu}} \cdot 2(2+\sqrt{10})LR,$$ it is sufficient to perform $$n = 4 \cdot \log_2\frac{MR\sqrt{2}}{\varepsilon} \cdot \log_\phi(\frac{C}{e\Delta})$$ operations to achieve accuracy in the function value of $\varepsilon$.
\end{theorem}

\begin{proof}
    By combining the theorem \ref{th:1} and the lemma \ref{l:1}, we obtain that in solving a one-dimensional problem, after $n_0 = \log_\phi(\frac{C}{e\Delta})$ operations, accuracy in the argument $\delta = \sqrt{\frac{2\phi\Delta\log_\phi(\frac{C}{\Delta})}{\mu}}$. is achieved. In the algorithm from the theorem \ref{th:2} the one-dimensional problem is solved 4 times for each iteration (2 times when transitioning from a square to a rectangle and 2 times from a rectangle to a square). The number of iterations performed is $k = \log_2\frac{MR\sqrt{2}}{\varepsilon}$. Therefore, the total number of operations will be $4kn_0 = 4 \cdot \log_2\frac{MR\sqrt{2}}{\varepsilon} \cdot \log_\phi(\frac{C}{e\Delta})$. Moreover, based on the theorem \ref{th:2}, to achieve accuracy $\varepsilon$ in the function value with this method, it is sufficient to satisfy the relationship $\delta \leqslant \frac{\varepsilon}{2(2+\sqrt{10})LR}$, which holds when $\varepsilon \geqslant \sqrt{\frac{2\phi\Delta\log_\phi(\frac{C}{\Delta})}{\mu}} \cdot 2(2+\sqrt{10})LR$.
\end{proof}

\section{Optimizing preference function}

We want to find the maximum of the individual preference function for a sweet and sour drink

One should note that this problem is a noisy one. Because a person's responses may depend on mood swings and other external factors.

Preference functions are likely to be concave \cite{lobanov2024order}. We used logarithmic coordinates because we expect from Stevens' law that the sensation of taste is a power function of the amount of ingredient \cite{stevens1957psychophysical,beebe1948general}.

The preference function maps the sourness and sweetness of the water onto the drink. $f: [1,4]\times[1,4] \to \mathbb{R}$. The first coordinate of the point $(x, y)$ corresponds to the level of acidity, and the second coordinate corresponds to the level of sweetness of water on a logarithmic scale. The exact formula for acidity: mass fraction of citric acid = $0.05\% * 3^x $. The exact formula for sweetness: mass fraction of sugar = $2\% * 2^y $. The boundary values were not chosen at random, but corresponded roughly to the limit of distinguishability or repugnance of taste. The relative mass fraction error in the preparation of water solutions was less than 5\%.

To find the position of the maximum of the function we will use the algorithm described in paragraph 5. Linear search via GRM operations will stop when a person cannot choose which drink they like more. After 4 iterations of linear search, the area of the region where the optimal point is found is reduced by 4 times. The algorithm reduced the search area 16 times in just 20 pairwise comparisons.


\newpage

\begin{figure} [H]
    \centering
    \includegraphics[width=\textwidth]{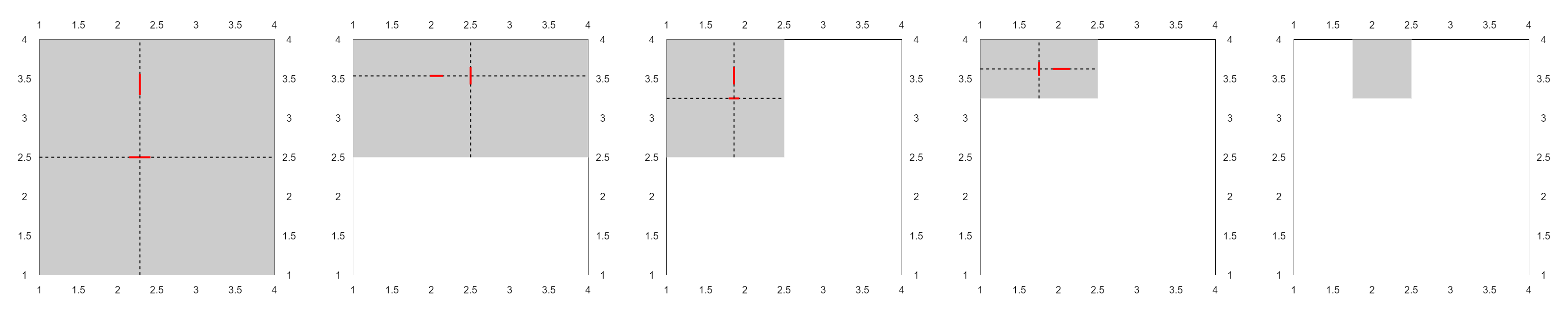}
    \caption{The progress of a squared algorithm for finding the maximum of the preference function of a sweet-and-sour drink. The x-axis is the acidity of water and the y-axis is its sweetness. The area where the optimal point is located is highlighted in gray. The dotted line shows the area where the linear search via GRM was performed. The bold red line connects two points where a person could not choose which drink they liked better. The search was stopped when the diameter of the square was of the same order of magnitude as the size of the red lines.}
    \label{fig1}
\end{figure}
The code for the experiments can be found at the link: \url{https://colab.research.google.com/drive/1jNX7_Naag-cL4D2lsmetV0B-3Q6GX1l3?usp=sharing}

\section{Conclusion}
In this study, the Nesterov's method \cite{пасечнюк2019одном} was modified to work with an Order Oracle: there is no longer a need to compute the subgradient or function values at individual points. Theoretical estimates accounting for noise were obtained.

The method for estimating the increase in the function minimum when transitioning to a smaller region performed well in finding the asymptotics for both one-dimensional and two-dimensional problems, showing approximately the same results as those for first-order oracle problems. Although not explicitly addressed in the study, similar reasoning can naturally be extended to higher dimensions.

In practice, the method converged faster than the theoretical estimates anticipated. The algorithm described in the study is suitable for real-world problems where only function values at points can be compared, but numerical function values at each individual point cannot be determined.

\subsubsection{\ackname} The research is supported by the Ministry of Science and Higher Education of the Russian Federation (Goszadaniye) No. 075-03-2024-117, project No. FSMG-2024-0025.

\bibliographystyle{splncs04}
\bibliography{biboptim}

\end{document}

%% file: macros.tex
\usepackage{amsfonts}
\usepackage{amsmath,amssymb}
\usepackage{algorithm}
\usepackage{pifont}
\usepackage{float}
\usepackage{graphicx}
\usepackage{caption}
\usepackage{subcaption}
\usepackage{blindtext}
\usepackage{xcolor}
\usepackage{booktabs}

\usepackage{colortbl}
\usepackage{color}
\usepackage{multirow}
\usepackage{pifont}
\usepackage[many]{tcolorbox}

\usepackage[font=small]{caption}
\usepackage[font=small]{subcaption}

\usepackage[algo2e]{algorithm2e}